\documentclass[oneside,10pt]{article}          
\usepackage[b5paper]{geometry}      
\usepackage{amsfonts,amsmath,latexsym,amssymb} 
\usepackage{amsthm}                
\usepackage{mathrsfs,upref}         
\usepackage{mathptmx}           
\usepackage{indentfirst}
\newtheorem{lemma}{Lemma}[section]
\newtheorem{corollary}{Corollary}[section]

\newtheorem{theorem}{Theorem}[section]         

\theoremstyle{definition}

\renewenvironment{proof}{\textbf{Proof:}}{$\Box$}

\numberwithin{equation}{section}

\begin{document}

\title{
Existence of optimizers of the Stein-Weiss inequalities on Carnot groups
\thanks{This work was supported by the National
Natural Science Foundation of China (Grant No. 11271299), Natural Science Foundation
Research Project of Shaanxi Province (Grant No. 2012JM1014). Corresponding Author: Pengcheng Niu, pengchengniu@nwpu.edu.cn}
}


\author{Tingxi Hu, Pengcheng Niu\\Department of Applied Mathematics,\\Northwestern Polytechnical University, Xi'an 710129, China}





\maketitle
\begin{abstract}
This paper proves existence of optimizers of the Stein-Weiss inequalities on Carnot groups under some conditions.
The adjustment of Lions' concentration compactness principles to Carnot groups plays an important role in our proof.
Unlike known treatment to the Hardy-Littlewood-Sobolev inequality on Heisenberg group,
our arguments relate to the powers of the weight functions.


\end{abstract}



\section{Introduction}

\subsection{Classic Hardy-Littlewood-Sobolev inequality and Stein-Weiss inequality}
     The well known Hardy-Littlewood-Sobolev inequality on $\mathbb{R}^N$ (short for HLS inequality) is of the form
\begin{equation}
\left| {\iint_{\mathbb{R}^N  \times \mathbb{R}^N } {\frac{{\overline {f\left( x
\right)} g\left( y \right)}} {{\left| {x - y} \right|^\lambda
}}dxdy}} \right| \leqslant C_{r,\lambda ,N} \left\| f \right\|_r
\left\| g \right\|_s,
\end{equation}
where $1 < r,s < \infty ,0 < \lambda  < N$ and $\frac{1} {r} +
\frac{1} {s} + \frac{\lambda } {N} = 2$.
$C_{r,\lambda ,N}$ is a positive constant independent of $f$ and $g$, $\left\| f \right\|_p$ denotes the $L^p(\mathbb{R}^N)$
norm of $f$.

This inequality was first proved by Hardy and Littlewood [5, 6] in
$\mathbb{R}^1$ and extended by Sobolev [15] to $\mathbb{R}^N$.
Lieb [11] classified all the optimizers of (1.1) on
$\mathbb{R}^N$ and obtained the sharp constant of this inequality
in the special case $ r = s = 2N/(2N - \lambda )$. Existence of optimizers of (1.1)
was also investigated by Lions [13], which is an application of the
concentration compactness principle.

The weighted HLS inequality, i.e. Stein-Weiss inequality, derived by
Stein and Weiss [7] on $\mathbb{R}^N$ which reads
\begin{equation}
\left| {\iint_{R^N  \times R^N } {\frac{{\overline {f\left( x
\right)} g\left( y \right)}} {{\left| x \right|^\alpha  \left| {x -
y} \right|^\lambda  \left| y \right|^\beta  }}dxdy}} \right|
\leqslant C_{\alpha ,\beta ,r,\lambda ,N} \left\| f \right\|_r
\left\| g \right\|_s ,
\end{equation}
where $1 < r,s < \infty ,0 < \lambda  < N,\alpha  + \beta \geqslant
0,\lambda  + \alpha  + \beta  \leqslant N,\alpha  < N/r'$ ($ r' =
r/(r - 1)$), and $\beta < N/s'$ ($s' = s/(s - 1)$), such that $
\frac{1} {r} + \frac{1} {s} + \frac{\lambda } {N} = 2$, $C_{\alpha
,\beta ,r,\lambda ,N}$ is a positive constant independent of $f,g$.

Recently, Han, Lu and Zhu [7] gave two classes of Stein-Weiss inequalities
on the Heisenberg group and claimed that these inequalities hold
in stratified groups. The authors [9] built the Stein-Weiss type
inequalities on Carnot groups, which supports the opinion of Han, Lu and Zhu.
Readers can also see [10] for results concerned.

Han and Niu [8] derived existence
of optimizers of the Hardy-Sobolev inequalities on the H-type group, which
applied a generalization of Lions' concentration compactness principles. On the other
hand, Han [4] furnished a proof of existence of optimizers of the
HLS inequality on the Heisenberg group in which the concentration
compactness principle plays an important role too. These inspire us to consider the related problems on
the Stein-Weiss inequalities on Carnot groups.

\subsection{Structure of Carnot group and the Stein-Weiss inequality}

We begin by describing Carnot group. For more information, we refer
to [2, 3, 14]. A Carnot group $G$ of step $r$ is a simply
connected nilpotent Lie group such that its Lie algebra
$\mathfrak{g}$ admits a stratification
$$ \mathfrak{g }= V_1  \oplus V_2 \oplus \ldots  \oplus V_r  =  \oplus _{l = 1}^r V_l,$$
in which $ [V_1 ,V_l ] = V_{l + 1} \;\;\left( {l = 1,2, \ldots ,r
-1} \right)$ and $ [V_1 ,V_r ] = \left\{ 0 \right\}$.

Denoting $ m_l  = \dim V_l$,we fix on $G$ a system of coordinates $u
= \left( {z_1 ,z_2 , \ldots ,z_r } \right), z_l  \in \mathbb{R}^{m_l }. $
Every Carnot group $G$ is naturally equipped with a family of
non-isotropic dilations
$$\delta _r \left( u \right) = \left( {rz_1 ,r^2 z_2 , \ldots ,r^r z_r
} \right)\;\;,\forall u \in G,\forall r > 0,$$ the homogeneous
dimension of $G$ is given by $Q = \sum\limits_{l = 1}^r {lm_l }$. We
denote by $du$ a fixed bi-invariant Haar measure on $G$. One easily
sees $ \left( {d \circ \delta _r } \right)\left( u \right) = r^Q du.$
The group law given by Baker-Campbell-Hausdorff formula is
$$uv = u + v + \sum\limits_{1 \leqslant l,k \leqslant r} {Z_{l,k}
\left( {u,v} \right)} \;,\;\;\forall u,v \in G,$$ where each
$Z_{l,k} \left( {u,v} \right)$ is a fixed linear combination of
iterated commutators containing $l$ times $u$ and $k$ times $v$. The
homogenous norm of $u$ on $G$ is defined by $$ \left| u \right| =
\left( {\sum\limits_{j = 1}^r {\left| {z_j } \right|^{\frac{{2r!}}
{j}} } } \right)^{\frac{1} {{2r!}}} ,$$ where $ \left| {z_j }
\right|$ denotes the Euclidean distance from $ z_j  \in \mathbb{R}^{m_j }$ to
the origin in $\mathbb{R}^{m_j }$. Such homogenous norm on $G$ can be used to define a
pseudo-distance $d\left( {u,v} \right) = \left|
{u^{ - 1} v} \right|$ on $G$.

Denote the pseudo-ball of radius $r$ centered at $u$ by $ B\left( {u,r}
\right) = \left\{ {v \in G|d\left( {u,v} \right) < r} \right\} ,$
and the pseudo-ball centered at the origin by $ B_r $ or $\left\{ {\left|
u \right| < r} \right\} $.

The Stein-Weiss type inequality on the Carnot group $G$ states as below (see [9] or [10]).

Let $ 1 < r,s < \infty ,0 < \lambda  < Q, \alpha  + \beta \geqslant
0, \alpha  < Q/r',(r' = r/(r - 1)), \lambda  + \alpha  + \beta
\leqslant Q, $ and $ \beta  < Q/s',(s' = s/(s - 1)), $ such that $
\frac{1} {r} + \frac{1} {s} + \frac{{\lambda  + \alpha  + \beta }}
{Q} = 2$, then there exists a positive constant $C_{\alpha ,\beta
,r,\lambda ,G} $ independent of $f,g$ such that
\begin{equation}
\left| {\iint_{G \times G} {\frac{{\overline {f\left( u \right)}
g\left( v \right)dudv}} {{\left| u \right|^\alpha  \left| {u^{ - 1}
v} \right|^\lambda  \left| v \right|^\beta  }}}} \right| \leqslant
C_{\alpha ,\beta ,r,\lambda ,G} \left\| f \right\|_r \left\| g
\right\|_s \;,
\end{equation}
where $ u = (z_1 ,z_2 , \ldots ,z_r ), v = (z'_1 ,z'_2 , \ldots
,z'_r ) \in G. $

\subsection{Main results}
The aim of this paper is to observe existence of optimizers of (1.3). By the dual argument (see [7]),
it is easy to get an alternative version of (1.3):

Let $ 1 < p \leqslant q < \infty ,0 < \lambda  < Q,\alpha  + \beta
\geqslant 0,\alpha  < Q/q,$ and $\beta  < Q/p',$ such that $
\frac{1} {q} = \frac{1} {p} +\frac{{\lambda  + \alpha  + \beta }}
{Q} - 1, $ then
\begin{equation}
\left\| {Sg} \right\|_q  \leqslant C\left\| g \right\|_p ,
\end{equation}
where $p=s, q=r', Sg\left( u \right): = \int_G {\frac{{g\left( v \right)}} {{\left| u
\right|^\alpha  \left| {u^{ - 1} v} \right|^\lambda  \left| v
\right|^\beta  }}dv},C = C_{\alpha ,\beta ,p,\lambda ,G}$ is a
positive constant independent of $g$.

Obviously, if we find an optimizer of (1.4), then we obtain
an optimizer of (1.3). But since (1.3) has the weight function $ |u|^\alpha $ and $
|v|^\beta $, so we should consider the range of $\alpha $ and
$\beta$, which is different from [4].

For $1 < p \leqslant q < \infty ,0 < \lambda  < Q, \alpha  + \beta \geqslant
0 $, we shall assume that

(H1):$\beta  < Q/p',\alpha  \leqslant 0$,

(H2):$ - Q/p < \beta  < Q/p',0 < \alpha  < \min \{ (\frac{{Q -
\lambda }} {Q})(\frac{Q} {{p'}} - \beta ),Q/q\},$ \\such that $ \frac{1} {q} = \frac{1} {p} + \frac{{\lambda  + \alpha  + \beta }}
{Q} - 1 $. Therefore, let us consider the constraint maximum problem, i.e.,
find optimizers of the functional $\left\| {Sf} \right\|_q
$:
\begin{equation}
C_0  = \mathop {\sup }\limits_{\left\| f \right\|_p  = 1} \left\|
{Sf} \right\|_q ,
\end{equation}
under the constraint $ \left\| f \right\|_p  = 1 $.

The main result of this paper is:
\begin{theorem}
Let $ \{ f_j \} $ be maximizing sequence of (1.5), then there exists $\{ u_j \}  \in G $ and $\{ d_j \}
\subset R_ + $ such that the following new maximizing sequence $\{ h_j \}$:
$$h_j (u): = \frac{1} {{d_j^{Q/p} }}f_j (\frac{{u_j u}} {{d_j }}),$$
is relatively compact in $ L^p (G) $, and the limitation of its
convergent subsequence is an optimizer of (1.5).
\end{theorem}

We will prove Theorem 1.1 by adjusting the concentration compactness principles on the
Euclidean space by Lions [12, 13] to one on the Carnot group $G$. If $G$ is
replaced with the Heisenberg group $\mathbb{H}^n$, then we have
\begin{corollary}

Let $ 1 < r,s < \infty ,0 < \lambda  < 2n + 2, 0 \leqslant \alpha +
\beta  \leqslant 2n + 2 - \lambda  $ and conditions bellow hold

(H1'):$\beta  < (2n + 2)/s',\alpha  \leqslant 0,$

(H2'):$ - (2n + 2)/s < \beta  < (2n + 2)/s',0 < \alpha  < \min \{
(\frac{{2n + 2 - \lambda }} {{2n + 2}})(\frac{{2n + 2}} {{s'}} -
\beta ),Q/r'\},$ \\such that $\frac{1} {r} + \frac{1} {s} +
\frac{{\lambda  + \alpha  + \beta }} {Q} = 2$,then there exists an
optimizer of the Stein-Weiss inequality on $\mathbb{H}^n$.

\end{corollary}

Notice that the condition (H1') in Corollary 1.1 contains the special
case $\alpha  = \beta  = 0 $, so the conclusion in Corollary 1.1 generalizes results in [4].

\section{Concentration compactness principles on Carnot groups}

\subsection{The first concentration compactness principle}

We state a lemma on the Carnot group $G$ which is actually true
in general measure spaces due to Br\'{e}zis and Lieb (see [1]).

\begin{lemma}
Let $0 < p < \infty, \{ f_j \}  \in L^p (G) $ satisfy $ f_j \to
f\;a.e.$, then
$$\mathop {\lim }\limits_{j \to \infty } \int_G {\left| {|f_j (u)|^p
- |f(u) - f_j (u)|^p  - |f(u)|^p } \right|du}  = 0.$$

\end{lemma}

Let us introduce the first concentration compactness principle on
$G$. The principle on the Heisenberg group was given by Han in [4].
The original version can see Lions [12].

\begin{lemma}
Let $\rho _j  = |f_j |^p du $ be a nonnegative Haar measure on $G$
 with $\int_G {\rho _j }  = 1$, then there exists a subsequence of
 $\{ \rho _j \}$ (still denoted by $ \{ \rho _j \}$) such that one of the following
 holds:

(1)For all $R>0$, we have
$$
\mathop {\lim }\limits_{j \to \infty } \left( {\mathop {\sup
}\limits_{u \in G} \int_{B(u,R)} {\rho _j } } \right) = 0.
$$

(2)There exists $ \{ u_j \}  \subset G $ such that for each $
\varepsilon  > 0 $ small enough, we can find $R_0  > 0 $ with

$$\int_{B(u_j ,R_0 )} {\rho _j }  > 1 - \varepsilon , \forall j \in
N.
$$

(3) There exists $0 < k < 1$ such that for each $ \varepsilon  > 0 $
small enough, we can find $ R_0  > 0 $ and $ \{ u_j \}  \subset G$
such that given any $ R \geqslant R_0 $, there exist $\rho _j^1 $
and $ \rho _j^2 $ satisfying

(a)$\rho _j^1  + \rho _j^2  = \rho _j$,

(b)$supp(\rho _j^2 ) \subset \left( {B(u_j ,R)} \right)^C $,

(c)$\mathop {\lim \sup }\limits_{j \to \infty } \left( {\left| {k -
\int_G {\rho _j^1 } } \right| + \left| {(1 - k) - \int_G {\rho _j^2
} } \right|} \right) \leqslant \varepsilon. $
\end{lemma}

Its proof is omitted, since it is similar to [4] without any new
difficult except replacing the Heisenberg group $\mathbb{H}^n$ by the Carnot group $G$.

Now let us define the Levy concentration function for $\rho _j $ on
$G$ by
$$
P_j (R) = \mathop {\sup }\limits_{u \in G} \int_{B(u,R)} {\rho _j }
, \;for\; any \;R \in [0,\infty ].$$ \\It is obvious that $ P_j  \in
BV[0,\infty ]$ is nonnegative and non-decreasing with
$$P_j (0) = 0,P_j (\infty ) = 1,\;for \; any\; j \in N. $$
\\Therefore, we can take a nonnegative and non-decreasing function
$P \in BV[0,\infty ] $ such that $P$ is a limit of some subsequence
of $\{ P_j \} $ (still denoted the subsequence by $\{ P_j \} $):
$$\mathop {\lim }\limits_{j \to \infty } P_j (R) = P(R),\;for\;any\; R \in
[0,\infty ).$$
\\Denote
$$k = \mathop {\lim }\limits_{R \to \infty } P(R), $$
\\thus $0 \leqslant k \leqslant 1.$ The case (1) of Lemma 2.2 holds if
$ k = 0 $; the case (2) holds if $k=1$ and the case (3)
holds for $0 < k < 1 $.

Let $ \{ f_j \}$ be a maximizing sequence of (1.5) satisfying $ \left\| {f_j }
\right\|_p  = 1. $ Lemma 2.2 ensures that one of the three cases
must happen. Using dilations in $G$ and choosing $ d_j $ large enough, we
can make a new maximizing sequence (still denoted by $ \{ f_j \}$)
such that
$$P_j (1) = \mathop {\sup }\limits_{u \in G} \int_{B(u)} {\rho _j }  =
\frac{1} {2}. $$
\\Then the case (1) of lemma 2.2 cannot occur.

The following result for $\alpha$ and $\beta$ such that
(1.4) holds is needed.

\begin{lemma}
Let $\{ f_j \}  \subset L^p (G)$ be a maximizing sequence of (1.5)
satisfying $\left\| {f_j } \right\|_p  = 1 ,$ then (3) of Lemma
2.2 cannot occur.
\end{lemma}
\begin{proof}
If the case (3) in Lemma 2.2 occurs, then there exist $ 0 < k < 1$
and a subsequence of $ \{ f_j \}$ (still denoted by $ \{ f_j \}$)
such that for each $ \varepsilon  > 0 $, one can find $ R_0  > 0$
and $ \{ u_j \}  \subset G $ such that for any $ R \geqslant R_0 ,$
\begin{eqnarray*}
   &&\left\| {f_j \chi _{B(R)} } \right\|_p^p  = k + O(\varepsilon ),\\
   &&\left\| {f_j \chi _{\left( {B(R)} \right)^C } } \right\|_p^p  = 1 -
k + O(\varepsilon ).
\end{eqnarray*}
Without loss of generality, we may assume $u_j  = 0,j \in N $ since
(1.5) is translation-invariant. For any $u \in G$, choose $ i
\geqslant i(\varepsilon ,|u|) $ such that $ i|u| > R_0 $ and let $ R
= i|u| $. We observe that $|u| \leqslant \frac{1} {i}|v|$ for all $v
\in \left( {B(R)} \right)^C $, then
\begin{equation*}
    |u^{ - 1} v| \geqslant |v| - |u| \geqslant \frac{{i - 1}} {i}|v|.
\end{equation*}
A direct calculation gives
\begin{eqnarray*}
   & & \left| {S(f_j )(u) - S(f_j \chi _{B(R)} )(u)} \right| \\
   &=& \left| {S(f_j \chi _{\left( {B(R)} \right)^C } )(u)} \right|  \\
   &=& \left| {\int_G {\frac{{(f_j \chi _{\left( {B(R)} \right)^C } )(v)}}
{{|u|^\alpha  |u^{ - 1} v|^\lambda  |v|^\beta  }}dv} } \right| \\
  &\leqslant& \left[ {\int_{\left( {B(R)} \right)^C } {|f_j (v)|^p dv} } \right]^{1/p}
  \left[ {\int_{\left( {B(R)} \right)^C } {|u|^{ - \alpha p'} |u^{ - 1} v|^{ - \lambda p'} |v|^{ - \beta p'} dv} } \right]^{1/p'} \\
\end{eqnarray*}
\begin{eqnarray*}
   &\leqslant & C|u|^{ - \alpha } \left( {\frac{i}
{{i - 1}}} \right)^\lambda  \left( {\int_{\left( {B(R)} \right)^C } {|v|^{ - \lambda p' - \beta p'} dv} } \right)^{1/p'}  \\
   & \leqslant & C\left( {\frac{i}
{{i - 1}}} \right)^\lambda  |u|^{ - a} \left( {\int_{i|u|}^\infty  {r^{Q - \lambda p' - \beta p' - 1} dr} } \right)^{1/p'}  \\
   & \leqslant & C\left( {\frac{i}
{{i - 1}}} \right)^\lambda  |u|^{ - a} \left( {\frac{1} {{\lambda p'
+ \beta p' - Q}}} \right)^{1/p'} \left( {i|u|} \right)^{(Q - \lambda
p' - \beta p')/p'} ,
\end{eqnarray*}
in which $C$ depends only on $G$.

Since $ \frac{1} {q} + \frac{1} {{p'}} = \frac{{\lambda  + \beta +
\alpha }} {Q}$, it follows $ Q/p' - \lambda  - \beta  = \alpha  -
\frac{Q} {q} < 0 $ and
$$ S(f_j \chi _{B(R)} )\xrightarrow{{a.e.}}S(f_j ), as \;i \to \infty .$$
\\By applying lemma 2.1, we have
\begin{equation}
\left\| {S(f_j )} \right\|_q^q  = \left\| {S(f_j \chi _{B(R)} )}
\right\|_q^q  + \left\| {S(f_j \chi _{\left( {B(R)} \right)^C } )}
\right\|_q^q  + o(1),\;as\;i \to \infty.
\end{equation}
\\Since $\{ f_j \} $ maximizes (1.5), it implies that the left hand side of (2.1) goes
to $ C_0^q $ as $j \to \infty$ for a large $i$, while the right hand side of (2.1) satisfies
\begin{eqnarray*}
  &&\left\| {S(f_j \chi _{B(R)} )} \right\|_q^q  + \left\| {S(f_j \chi _{\left( {B(R)} \right)^C } )} \right\|_q^q  + o(1) \hfill \\
   &\leqslant& C_0^q \left\| {f_j \chi _{B(R)} } \right\|_p^q  + C_0^q \left\| {f_j \chi _{\left( {B(R)} \right)^C } } \right\|_p^q  + o(1) \hfill \\
   &\leqslant &C_0^q (k + O(\varepsilon ))^{\frac{q}
{p}}  + C_0^q (1 - k + O(\varepsilon ))^{\frac{q} {p}}  + o(1) \hfill \\
   &\leqslant& C_0^q \left( {k^{\frac{q}
{p}}  + (1 - k)^{\frac{q} {p}} } \right) + O(\varepsilon ) +o(1) < C_0^q , \hfill \\
\end{eqnarray*}
which is a contradiction.
\end{proof}

\subsection{The convergent subsequence of maximizing sequence}
Let $ \{ f_j \}$ be a maximizing sequence of (1.5), we see from the
argument in previous subsection that there exists $\{ u_j \} \subset G
$ such that for $R$ large enough
$$ \int_{B(u_j ,R)} {|f_j |^p }  \geqslant 1 - \varepsilon (R). $$
\\Translating $f_j(v)$ into $f_j(u_jv)$, we make a new maximizing sequence
$ \{ f_j \} $ satisfying
\begin{equation}
\int_{B(R)} {|f_j |^p }  \geqslant 1 - \varepsilon (R).
\end{equation}

Now let us prove that we can take a convergent subsequence of the maximizing sequence
of (1.5) by using (2.2).

\begin{lemma}
Let $\{ f_j \}  \subset L^p (G)$ be a maximizing sequence of (1.5)
satisfying $ \left\| {f_j } \right\|_p  = 1 $ and (2.2). Assume that
$f_j  \to f $ weakly in $L^p(G)$, then there exists a subsequence of
$\{ f_j \}$ ( still denoted by $\{f_j\}$ ) such that
\begin{equation}
S(f_j )\xrightarrow{{a.e.}}S(f).
\end{equation}
\end{lemma}
\begin{proof}
We show $S(f_j ) \to S(f)$ in measure to ensure existence
of a point-wisely convergent subsequence of $\{f_j\}$. A direct
computation yields
\begin{eqnarray}
    & &\left\| {S(f_j )\chi _{\left( {B(M)} \right)^C } } \right\|_q  \nonumber\\
   &\leqslant& \left\| {S(f_j \chi _{B(R)} )\chi _{\left( {B(M)} \right)^C } }
   \right\|_q  + \left\| {S(f_j \chi _{\left( {B(R)} \right)^C } )\chi _{\left( {B(M)} \right)^C } } \right\|_q  \nonumber\\
   &\leqslant &\left\| {S(f_j \chi _{B(R)} )\chi _{\left( {B(M)} \right)^C } }
   \right\|_q  + C_0 \left\| {f_j \chi _{\left( {B(R)} \right)^C } } \right\|_p  \nonumber\\
   &\leqslant& \left\| {S(f_j \chi _{B(R)} )\chi _{\left( {B(M)} \right)^C } } \right\|_q  + \varepsilon (R).
\end{eqnarray}
Notice that for $ M > R , v \in B(R),u \in \left( {B(M)} \right)^C$,
it follows $|u^{ - 1} v| \geqslant |u| - R.$ We apply Minkowski's
integral inequality to get
\begin{eqnarray}
  &&\left\| {S(f_j \chi _{B(R)} )\chi _{\left( {B(M)} \right)^C } } \right\|_q  \nonumber\\
   &=& \left( {\int_{|u| > M} {\left| {\int_{|v| \leqslant R} {\frac{{f_j (v)}}
{{|u|^\alpha  |u^{ - 1} v|^\lambda  |v|^\beta  }}dv} } \right|^q du} } \right)^{1/q}  \nonumber\\
   &\leqslant& \left( {\int_{|u| > M} {\frac{1}
{{|u|^{\alpha q} (|u| - R)^{\lambda q} }}du} } \right)^{1/q} \left|
{\int_{|v| \leqslant R} {\frac{{f_j (v)}}
{{|v|^\beta  }}dv} } \right| \nonumber\\
  &\leqslant &\left( {\int_M^{ + \infty } {\frac{{r^{Q - 1 - \alpha q - \lambda q} }}
{{(1 - R/r)^{\lambda q} }}dr} } \right)^{1/q} \left| {\int_{|v|
\leqslant R} {\frac{{f_j (v)}}
{{|v|^\beta  }}dv} } \right| \nonumber\\
 &\leqslant &\frac{C}
{{(1 - R/M)^\lambda  }}\left( {\int_M^{ + \infty } {r^{Q - 1 -
\alpha q - \lambda q} dr} } \right)^{1/q}\centerdot\nonumber
\\&&\left( {\int_{|v| \leqslant R} {|f_j (v)|^{t'} dv} }
\right)^{1/t'} \left( {\int_{|v| \leqslant R} {\frac{1} {{|v|^{\beta
t} }}dv} } \right)^{1/t} ,
\end{eqnarray}
where $t'$ denote the conjugate index of $t$ such that $t' < p $, $t< Q/\beta $.

Since $Q - \alpha q - \lambda q = q(\beta  - Q/p') < 0 $ and $f_j
\in L^p (G) \subset L_{loc}^{t'} (G) $, for every fixed $R$, $
M \gg R$, we have
\begin{equation}
\left\| {S(f_j )\chi _{\left( {B(M)} \right)^C } } \right\|_q
\leqslant \varepsilon (R).
\end{equation}
Similarly, one has
\begin{equation}
\left\| {S(f)\chi _{\left( {B(M)} \right)^C } } \right\|_q \leqslant
\varepsilon (R).
\end{equation}
Thus, for every $k>0$, it implies
\begin{eqnarray}
 &&\left| {\left\{ {|S(f_j ) - S(f)| \geqslant 15k} \right\}} \right| \nonumber\\
   &\leqslant& \left| {\left\{ {|S(f_j ) - S(f_j )\chi _{B(M)} | \geqslant 5k} \right\}} \right| + \nonumber \\
  &&\left| {\left\{ {|S(f_j )\chi _{B(M)}  - S(f)\chi _{B(M)} | \geqslant 5k} \right\}} \right| +  \nonumber\\
  &&\left| {\left\{ {|S(f)\chi _{B(M)}  - S(f)| \geqslant 5k} \right\}} \right| \nonumber\\
 &\leqslant &2\left[ {\frac{{\varepsilon (R)}}
{{5k}}} \right]^q  + \left| {\left\{ {|S(f_j ) - S(f)| \geqslant 5k}
\right\} \cap B(M)} \right|.
\end{eqnarray}

Now the remainder is to estimate $ \left| {\left\{ {|S(f_j ) - S(f)| \geqslant 5k}
\right\} \cap B(M)} \right|$. Noting
\begin{eqnarray}
  &&\left| {\left\{ {|S(f_j ) - S(f)| \geqslant 5k} \right\} \cap B(M)} \right|\nonumber \\
   &\leqslant& \left| {\left\{ {|S(f_j ) - S(f_j \chi _{B(R')} )| \geqslant k} \right\}}
   \right| + \left| {\left\{ {|S(f_j \chi _{B(R')} ) - S^\eta  (f_j \chi _{B(R')} )| \geqslant k} \right\}} \right| + \nonumber \\
 && \left| {\left\{ {|S^\eta  (f_j \chi _{B(R')} ) - S^\eta  (f\chi _{B(R')} )| \geqslant k} \right\} \cap B(M)} \right| + \nonumber \\
 && \left| {\left\{ {|S^\eta  (f\chi _{B(R')} ) - S(f\chi _{B(R')} )|
  \geqslant k} \right\}} \right| + \left| {\left\{ {|S(f\chi _{B(R')} )(u) - S(f)| \geqslant k} \right\}} \right|\nonumber\\
  &:=&J_1+J_2+J_3+J_4+J_5,
\end{eqnarray}
where
$$ S^\eta  (f)(u) = \int_{\left( {B(u,\eta )} \right)^C }
{\frac{{f(v)}} {{|u|^\alpha  |u^{ - 1} v|^\lambda  |v|^\beta  }}dv}
,R' > 0, $$
\\we estimate $ J_1 ,J_2 ,J_3 ,J_4 $ and $J_5$ respectively.

First notice that
\begin{equation}
S^\eta  (f_j \chi _{B(R')} )(u) \to S^\eta  (f\chi _{B(R')}
)(u),\forall u \in G.
\end{equation}
Because $|u|^{ - \alpha } |u^{ - 1} v|^{ - \lambda } |v|^{ - \beta }
\chi _{B(R')} \chi _{\left( {B(u,\eta )} \right)^C }  \in L^{p'}
(G), $ we observe
\begin{equation}
J_3=\left| {\left\{ {\left| {S^\eta  (f_j \chi _{B(R')} ) - S^\eta
(f\chi _{B(R')} )} \right| \geqslant k} \right\} \cap B(M)} \right|
= o(1),\;as\; j \to \infty.
\end{equation}

Since $f_j \to f$ weakly in $L^p(G)$,
we get
$$ \left\| {f\chi _{\left( {B(R')} \right)^C } } \right\|_p^p \leqslant
\mathop {\lim \inf }\limits_{j \to \infty } \left\| {f_j \chi
_{\left( {B(R')} \right)^C } } \right\|_p^p  \leqslant \varepsilon
^p (R'),$$
\\it yields
\begin{eqnarray}
J_1&\leqslant&\frac{1}{k^q}\left\| {S(f_j ) - S(f_j \chi _{B(R')} )} \right\|_q^q  \leqslant
 \frac{C_0}{k^q}\left\| {f_j \chi _{\left( {B(R')} \right)^C } } \right\|_p^q
\leqslant (\frac{ \varepsilon(R')}{k})^q,\\
J_5&\leqslant&\frac{1}{k^q}\left\| {S(f) - S(f\chi _{B(R')} )} \right\|_q^q  \leqslant \frac {C_0}{k^q}
\left\| {f\chi _{\left( {B(R')} \right)^C } } \right\|_p^q  \leqslant
(\frac{\varepsilon (R')}{k})^q.
\end{eqnarray}

To compute $J_2$ and $J_4$, we claim the following two statements:

1)When (H1) holds, there is $ m_1  \in (1,Q/\lambda ) $ such that for
a fixed $R'$,
\begin{equation}
\left\| {S(f_j \chi _{B(R')} ) - S^\eta  (f_j \chi _{B(R')} )}
\right\|_{m_2 }  \leqslant O(\eta ),\;\eta \to 0.
\end{equation}

2)When (H2) holds, there exists $ m_2  \in (1,Q/(\lambda  + \alpha )) $
such that for a fixed $R'$,
\begin{equation}
\left\| {S(f_j \chi _{B(R')} ) - S^\eta  (f_j \chi _{B(R')} )}
\right\|_{m_2 }  \leqslant O(\eta ),\; \eta \to 0.
\end{equation}
We first maintain them and give their proofs latter.
Choosing $ m = \chi _{( - \infty ,0]} (\alpha )m_1  + \chi _{(0, +
\infty )} (\alpha )m_2 , $ one has for the condition (H1) or (H2),
\begin{equation}
J_2 \leqslant \left\| {S(f_j \chi _{B(R')} ) - S^\eta  (f_j \chi _{B(R')} )}
\right\|_m  \leqslant O(\eta ).
\end{equation}
Similarly it follows for $f$,
\begin{equation}
J_4 \leqslant\left\| {S(f\chi _{B(R')} ) - S^\eta  (f\chi _{B(R')} )} \right\|_m
\leqslant O(\eta ).
\end{equation}

Substitute (2.11)-(2.13),(2.16) and (2.17) into (2.9) we have
\begin{eqnarray*}
  &&\left| {\left\{ {|S(f_j ) - S(f)| \geqslant 5k} \right\} \cap B(M)} \right| \\
   &&\leqslant 2\left[ {\frac{{\varepsilon (R')}}
{k}} \right]^q  + 2\left[ {\frac{{O(\eta )}} {k}} \right]^m  + o(1).
\end{eqnarray*}
It shows that $ \{ S(f_j )\} $ is convergent in measure, and then
$\{f_j\}$ is convergent in measure by properly choosing $
\varepsilon ,R,R'$ and $\eta$.

\emph{Proof of (2.14) and (2.15).} Concretely, we need to prove
(2.14) under the condition $ \beta  < Q/p'$ and $\alpha \leqslant
0 $, and (2.15) under the condition $- Q/p < \beta  < Q/p'$ and $0 < \alpha  < \min \{ (\frac{{Q - \lambda }}
{Q})(\frac{Q} {{p'}} - \beta ),Q/q\}$.

To prove (2.14), we choose $ m_1  \in (1,Q/\lambda ) $ and apply Minkowski's
integral inequality,
\begin{eqnarray*}
  &&\left\| {S(f_j \chi _{B(R')} ) - S^\eta  (f_j \chi _{B(R')} )} \right\|_{m_1 }  \\
   &=& \left( {\int_G {\left| {\int_{B(u,\eta )} {\frac{{f_j (v)\chi _{B(R')} (v)}}
{{|u|^\alpha  |u^{ - 1} v|^\lambda  |v|^\beta  }}dv} } \right|^{m_1 } du} } \right)^{1/m_1 }  \\
   &\leqslant& \int_G {\left( {\int_{B(v,\eta )} {\frac{{|f_j (v)\chi _{B(R')} (v)|^{m_1 } }}
{{|u|^{\alpha m_1 } |u^{ - 1} v|^{\lambda m_1 } |v|^{\beta m_1 } }}du} } \right)^{1/m_1 } dv}  \\
   &=& \int_G {\frac{{|f_j (v)\chi _{B(R')} (v)|}}
{{|v|^\beta  }}\left( {\int_{B(v,\eta )} {\frac{1}
{{|u|^{\alpha m_1 } |u^{ - 1} v|^{\lambda m_1 } }}du} } \right)^{1/m_1 } dv}  \\
   &\leqslant& \int_G {\frac{{|f_j (v)\chi _{B(R')} (v)|}}
{{|v|^\beta  }}\left( {\int_{B(\eta )} {\frac{1}
{{(\eta  + R')^{\alpha m_1 } |u|^{\lambda m_1 } }}du} } \right)^{1/m_1 } dv}  \\
   &\leqslant& C(R',p,n)(\eta  + R')^{ - \alpha } \eta ^{(Q - \lambda m_1 )/m_1 } \int_{B(R')} {\frac{{|f_j (v)|}}
{{|v|^\beta  }}dv} ,
\end{eqnarray*}
where $ \beta  < Q/p'$, $ \lambda m_1  < Q $. Therefore for a fixed
$R'$, as $\eta \to 0$, we get (2.14).

As to (2.15), we see
\begin{equation*}
\frac{{Q/p' - \beta }} {{(Q/p' - \beta ) - \alpha }} < \frac{Q}
{\lambda },\;(\lambda  + \alpha ) < Q,
\end{equation*}
and choose $m_2$ such that $ 1 < m_2  < Q/(\lambda  + \alpha ) $ and $ 1
< m_2  < \frac{Q} {\lambda } \cdot \frac{{(Q/p' - \beta ) - \alpha
}} {{Q/p' - \beta }}. $ From $ m_2  < Q/(\lambda  + \alpha ) $ we
get
\begin{equation*}
Q/(Q - \lambda m_2 ) < Q/\alpha m_2 .
\end{equation*}
Notice that $ m_2  < \frac{Q} {\lambda } \cdot \frac{{(Q/p' - \beta
) - \alpha }} {{Q/p' - \beta }} $ implies that
\begin{equation*}
\frac{Q} {{(Q - \lambda m_2 )}} < \frac{1} {\alpha }(\frac{Q} {{p'}}
- \beta ).
\end{equation*}
Thus , we can take $l$ such that $Q/(Q - \lambda m_2 ) < l < \min
\{ Q/\alpha m_2 ,\frac{1} {\alpha }(\frac{Q} {{p'}} - \beta )\}$. It
yields
\begin{equation*}
\alpha m_2 l < Q,\; \lambda m_2 l' < Q,\;\beta  + \alpha l < Q/p',
\end{equation*}
in which $ l' = l/(l - 1).$

Applying Minkowski's integral inequality to get

\begin{eqnarray}
  &&\left\| {S(f_j \chi _{B(R')} ) - S^\eta  (f_j \chi _{B(R')} )} \right\|_{m_2 } \nonumber \\
   &=& \left( {\int_G {\left| {\int_{B(u,\eta )} {\frac{{f_j (v)\chi _{B(R')} (v)}}
{{|u|^\alpha  |u^{ - 1} v|^\lambda  |v|^\beta  }}dv} } \right|^{m_2 } du} } \right)^{1/m_2 } \nonumber \\
   &\leqslant& \int_G {\left( {\int_{B(v,\eta )} {\frac{{|f_j (v)\chi _{B(R')} (v)|^{m_2 } }}
{{|u|^{\alpha m_2 } |u^{ - 1} v|^{\lambda m_2 } |v|^{\beta m_2 } }}du} } \right)^{1/m_2 } dv}  \nonumber\\
  & = &\int_G {\frac{{|f_j (v)\chi _{B(R')} (v)|}}
{{|v|^\beta  }}\left( {\int_{B(v,\eta )} {\frac{1}
{{|u|^{\alpha m_2 } |u^{ - 1} v|^{\lambda m_2 } }}du} } \right)^{1/m_2 } dv}  \nonumber\\
   &\leqslant& \int_G {\frac{{|f_j (v)\chi _{B(R')} (v)|}}
{{|v|^\beta  }}\left( {\int_{B(v,\eta )} {\frac{1} {{l|u|^{\alpha
m_2 l} }} + \frac{1}
{{l'|u^{ - 1} v|^{\lambda m_2 l'} }}du} } \right)^{1/m_2 } dv}  \nonumber\\
   &\leqslant& \int_G {\frac{{|f_j (v)\chi _{B(R')} (v)|}}
{{|v|^\beta  }}\left( {\int_{B(v,\eta )} {\frac{1}
{{l|u|^{\alpha m_2 l} }}du} } \right)^{1/m_2 } dv}  +  \nonumber\\
 && \int_G {\frac{{|f_j (v)\chi _{B(R')} (v)|}}
{{|v|^\beta  }}dv} \left( {\int_{B(\eta )} {\frac{1}
{{l'|u|^{\lambda m_2 l'} }}du} } \right)^{1/m_2 }  \nonumber\\
&: =& I_1  + I_2 .
\end{eqnarray}

To estimate $I_1$, noting $ \alpha m_2 l < Q $, $|u|^{ - \alpha m_2
l}  \in L_{loc} (G) $, and applying Lebesgue's integral theorem, we have that as
$\eta \to 0$,
\begin{equation*}
\frac{1} {{|B(v,\eta )|}}\int_{B(v,\eta )} {\frac{1} {{|u|^{\alpha
m_2 l} }}du}  = \frac{1} {{|v|^{\alpha m_2 l} }} + O(\eta ^{m_2 } ).
\end{equation*}
Thus
\begin{eqnarray*}
  I_1  &\leqslant &\int_G {\frac{{|f_j (v)\chi _{B(R')} (v)|}}
{{|v|^\beta  }}\left( {\frac{{C\eta ^Q }} {{|B(v,\eta
)|}}\int_{B(v,\eta )} {\frac{1}
{{l|u|^{\alpha m_2 l} }}du} } \right)^{1/m_2 } dv}  \\
  & \leqslant& C\frac{{\eta ^{Q/m_2 } }}
{{l^{1/m_2 } }}\int_G {\frac{{|f_j (v)\chi _{B(R')} (v)|}}
{{|v|^\beta  }}\left( {\frac{1} {{|B(v,\eta )|}}\int_{B(v,\eta )}
{\frac{1}
{{|u|^{\alpha m_2 l} }}du} } \right)^{1/m_2 } dv}  \\
   &\leqslant& C\frac{{\eta ^{Q/m_2 } }}
{{l^{1/m_2 } }}\int_G {\frac{{|f_j (v)\chi _{B(R')} (v)|}}
{{|v|^\beta  }}(\frac{1}
{{|v|^{\alpha l} }} + O(\eta ))dv}  \\
   &\leqslant &C\frac{{\eta ^{Q/m_2 } }}
{{l^{1/m_2 } }}\int_{B(R')} {\frac{{|f_j (v)|}} {{|v|^{\beta  +
\alpha l} }}dv}  + O(\eta ).
\end{eqnarray*}
Since $ f_j  \in L^p $, $ \beta  + \alpha l < Q/p'$, it follows
\begin{equation}
I_1  = O(\eta ).
\end{equation}

Since $ \beta  < Q/p' $, $\lambda m_2 l' < Q $, we have that as $\eta \to 0$,
\begin{eqnarray}
I_2  &\leqslant& (\int_{B(R')} {|f_j (v)|^p dv})^{1/p}
(\int_{B(R')}{\frac{1}{{|v|^{\beta{p'}}  }}dv})^{1/p'}
\left( {\int_{B(\eta )} {\frac{1}
{{l'|u|^{\lambda m_2 l'} }}du} } \right)^{1/m_2 }\nonumber \\
&\leqslant &O(\eta ).
\end{eqnarray}

Combining with (2.18)-(2.20), it yields (2.15).
\end{proof}

\subsection{The second concentration compactness principle}
Now we have found a weakly convergent subsequence of the maximizing
sequence of (1.5). To complete the proof of Theorem 1.1, it wants to
prove that this subsequence converges strongly. For the purpose, we
need the second concentration compactness principle on $G$, which is
a special case of known results in measure spaces, see [13]. Here it is a description in $G$.
\begin{lemma}
Let $f_j  \to f$ weakly in $ L^p (G)$, $S(f_j ) \to S(f)$ weakly in
$L^p(G)$. Assume that (2.2) and (2.6) hold, and the nonnegative
measures $\rho _j  \to \mu $ weakly in $L(G)$, $ |S(f_j )|^q du \to
\nu$ in $L(G)$. Then, there exist two at most countable families $
\{ u_j \}  \subset G,\{ k_j \}  \subset (0,\infty ) $ such that
\begin{gather*}
\nu  = |S(f)|^q du + \sum\limits_j {C_0 k_j^{q/p} \delta _{u_j } },\\
\mu  \geqslant |f|^p du + \sum\limits_j {k_j \delta _{u_j } },
\end{gather*}
in which $\delta _{u_j }$ is the Dirac measure at $u_j$.
\end{lemma}

\section{Proof of Theorem 1.1}
\begin{proof}
Let $ \{ f_j \}$ be the maximizing sequence in (1.5) satisfying
(2.2), $ f_j  \to f $ weakly in $ L^p (G) $ and condition (2.3)
hold. We will show $ \left\| f \right\|_p  = 1 $. Let us notice that
$ \mu (G) = 1,\nu (G) = C_0^q$. If $\left\| f \right\|_p^p  = k <
1$, then
\begin{equation*}
\sum\limits_j {k_j }  \leqslant \mu (G) - \left\| f \right\|_p^p  =
1 - k.
\end{equation*}
Therefore,
\begin{eqnarray*}
  \nu (G) &=& \left\| {S(f)} \right\|_q^q  + \sum\limits_j {C_0^q k_j^{q/p} }  \\
   &\leqslant &C_0^q \left\| f \right\|_p^p  + C_0^q (\sum\limits_j {k_j } )^{q/p}  \\
   &\leqslant &C_0^q k^{q/p}  + C_0^q (\sum\limits_j {k_j } )^{q/p}  \\
   &\leqslant &C_0^q k^{q/p}  + C_0^q (1 - k)^{q/p}  \\
   &<& C_0^q ,
\end{eqnarray*}
which contradicts with the fact that $ \nu (G) = C_0^q $ and we
complete the proof of Theorem 1.1.
\end{proof}

\section{Conflict of Interests}
The authors declare that there is no no conflict of
interests regarding the publication of this article.



\begin{thebibliography}{15}

\bibitem{1}
        {\sc H. Br\'{e}zis and E. Lieb},
        {\it A relation between pointwise convergence of functions and convergence of functional [J]},
         Proc Amer Math Soc, {\bf 88}, 3 (1983), 486--490.

\bibitem{2}
        {\sc A. Bonfiglioli, F. Uguzzoni and E. Lanconelli},
        {\it Stratified Lie groups and potential theory for their sub-Laplacians [M]},
        Springer Monogr Math, Springer-Verlag, Berlin, (2007).
\bibitem{3}
        {\sc N. Garofalo and D. Vassilev},
        {\it Regularity near the characteristic set in the non-linear Dirichlet problem and conformal geometry of sub-Laplacians on Carnot groups [J]},
         Math Ann, {\bf 318}, 3 (2000), 453--516.
\bibitem{4}
        {\sc X. Han},
        {\it Existence of maximizers for Hardy-Littlewood-Sobolev inequalities on the Heisenberg group [J]},
        To appear in Indiana University Mathematics Journal.
\bibitem{5}
        {\sc G. H. Hardy and J. E. Littlewxood},
        {\it Some properties of fractional integrals (1) [J]},
         Math Z, {\bf 27}, 1 (1928), 565--606.
\bibitem{6}
        {\sc G. H. Hardy and J. E. Littlewood},
        {\it On certain inequalities connected with the calculus of variations [J]},
        J London Math, {\bf 5}, 1 (1930), 34--39.
\bibitem{7}
        {\sc X. Han, G. Lu and J. Zhu},
        {\it Hardy-Littlewood-Sobolev and Stein-Weiss inequalities and integral systems on the Heisenberg group [J]},
         Nonlinear Analysis, {\bf75}, 11 (2012), 4296--4314.
\bibitem{8}
        {\sc Y. Han and P. Niu},
        {\it Hardy-Sobolev type inequalities on the H-type group [J]},
        Manuscripta Math, {\bf 118}, 2 (2005), 235--252.
\bibitem{9}
        {\sc T. Hu and P. Niu},
        {\it Hardy-Littlewood-Sobolev Type Inequality and Stein-Wiess Type Inequality on Carnot Groups},
        arXiv:1303.5185.
\bibitem{10}
        {\sc V. Kokilashvili and A. Meskhi},
        {\it On some weighted inequalities for fractional integrals on nonhomogeneous spaces [J]},
        Z Anal Anwendungen, {\bf  24}, 4 (2005), 871--885.
\bibitem{11}
        {\sc E. H. Lieb},
        {\it Sharp constants in the Hardy-Littlewood-Sobolev and related inequalities [J]},
        Ann of Math, {\bf 118}, 2 (1983), 349--374.
\bibitem{12}
        {\sc P. Lions},
        {\it The concentration-compactness principle in the calculus of variations. The locally compact case. I [J]},
         Ann IHP, Anal Nonlin, {\bf 1}, 2 (1984), 109--145.
\bibitem{13}
        {\sc P. Lions},
        {\it The concentration-compactness principle in the calculus of variations. The limit case. II [J]},
          Rev Mat Iberoamericana, {\bf 1}, 2 (1985), 45--121.
\bibitem{14}
        {\sc  F. Ricci},
        {\it Sub-Laplacians on Nilpotent Lie Groups. Course notes [M]},
        posted at:
        http://cvgmt.sns.it/math/Ricci/corsi.html
\bibitem{15}
        {\sc S. L. Sobolev},
        {\it On a theorem of functional analysis [J]},
         Math Soc Transl, {\bf 34}, 2 (1963), 39--68.

\end{thebibliography}
\end{document}